\documentclass[11pt]{article}
\usepackage{geometry}
\usepackage{amssymb}
\usepackage{latexsym}
\usepackage{amsmath}
\usepackage{indentfirst}
\usepackage{graphicx}
\usepackage[colorlinks=true]{hyperref}
\usepackage{mathrsfs}
\usepackage{color}
\usepackage{extarrows}

\usepackage{tikz}

\newtheorem{Thm}{Theorem}[section]

\newtheorem{Lem}{Lemma}[section]
\newtheorem{Pro}{Proposition}[section]

\newcommand{\R}{\mathbb{R}}

\numberwithin{equation}{section} \numberwithin{figure}{section}

\newenvironment{proof}{\medskip\par\noindent{\bf Proof\/}.\quad}{\qquad
\raisebox{-0.5mm}{\rule{2.5mm}{2.5mm}}\vspace{7pt}}

\begin{document}
\title{On nonlinear elliptic  problems with Hardy-Littlewood-Sobolev critical exponent and Sobolev-Hardy critical exponent}
\author{\quad Guangze Gu$^{1,2}$,\quad Aleks Jevnikar$^{1}$ \\
\footnotesize{\em
1.  Department of Mathematics, Computer Science and Physics, University of Udine, Via delle }\\
\footnotesize{\em
 Scienze 206, 33100 Udine, Italy.}\\
\footnotesize{\em
2. Department  of  Mathematics, Yunnan  Normal  University, Kunming 650500, China. }\\
\footnotesize{Email: guangzegu@163.com, aleks.jevnikar@uniud.it}
}

\date{} \maketitle
\begin{abstract}
In this paper we use variational methods to establish the existence of solutions for a class of nonlinear elliptic problems involving a combined convolution-type and Hardy nonlinearity with subcritical and critical growth.

\end{abstract}

\vspace{6mm} \noindent{\bf Keywords:} Choquad equation; Critical exponent; variational method.

\vspace{6mm} \noindent
{\bf 2010 Mathematics Subject Classification.} 35A15, 35J20, 35R11, 47G20.

\section{Introduction}
Let $\Omega \subset \mathbb{R}^N$ be a smooth open bounded  set containing  $0$ in its interior. In this paper we focus on studying the existence of solutions to the following nonlocal problem
\begin{equation}\label{eq1.1}
\left\{\begin{array}{ll}
-\Delta u(x)=\lambda
\Big( \displaystyle{\int_{\Omega}} \frac{|u(y)|^{p} }{|x-y|^{\alpha}}dy\Big) |u(x)|^{p-2}u(x)+\mu\frac{ |u|^{q-2}u}{|x|^s},     \,  &x\in  \Omega,  \\
u(x) =0,    \, &  x\in \partial \Omega,
\end{array} \right.
\end{equation}
where $0<\alpha<N$, $\lambda, \mu>0$, $0\leq s\leq2$, $1< p\leq2_{\alpha}^*=\frac{2N-\alpha}{N-2}$, $2\leq q\leq 2^*(s)=\frac{2(N-s)}{N-2}$. Here $2_{\alpha}^*$ is usually called the upper critical exponent in the sense of the Hardy-Littlewood-Sobolev inequality, as is clearly shown by the following estimate:
\begin{equation*}\label{eq2.2}
\int_{\R^N}\int_{\R^N} \frac{|u(x)|^{2_{\alpha}^*}|u(y)|^{2_{\alpha}^*} }{|x-y|^{\alpha}}dxdy \leq C(N,\alpha) ||u||_{2^*}^{ 2\cdot2_{\alpha}^* }<+\infty,
\end{equation*}
for all $u \in D^{1,2}(\mathbb{R}^N)$,  where $C(N,\alpha)= \pi^{\frac{\alpha}{2}} \frac{\Gamma(\frac{N-\alpha}{2})} {\Gamma(\frac{2N-\alpha}{2})} \bigg(\frac{\Gamma(\frac{N}{2})} {\Gamma(N)} \bigg)^{\frac{\alpha}{N}-1}$. The Hardy-Sobolev critical exponent $2^*(s)$ arises from the following Hardy-Sobolev inequality (see, e.g. \cite{Ghoussoub-Yuan2000TAMS})
\[ \Big( \int_{\mathbb{R}^N} \frac{|u|^{2^*(s)}}{|x|^s}dx\Big)^{\frac{2}{2^*(s)}} \leq C(N,s) \int_{\mathbb{R}^N}|\nabla u|^2dx,~~\text{for all}~ u \in D^{1,2}(\mathbb{R}^N), \]
which is essentially due to Caffarelli, Kohn and Nirenberg \cite{Caffarelli-Kohn-Nirenberg1984CM}. Note that $2^*(0)$ is nothing but the Sobolev critical exponent.

The investigation of \eqref{eq1.1} stems from the following nonlocal Choquard equation:
\begin{equation}\label{eq1.2}
-\Delta u+u=\Big( \int_{\mathbb{R}^N}\frac{|u(y)|^p}{ |x-y|^\alpha}dy \Big)|u(x)|^{p-1},~~x\in\mathbb{R}^N,
\end{equation}
where $0<\alpha<N$.
Equation \eqref{eq1.2} is closely related to the Choquard equation, which arises in the study  of Bose-Einstein condensation and is used to model the finite range many-body interactions between particles.
For $N=3$ and $\alpha=1$, equation \eqref{eq1.2}  becomes  the Choquard-Pekar equation which was introduced by Pekkar  \cite{Pekar1954} in 1954 to describe the quantum theory of a stationary polaron at rest. Later in 1976, Choquard \cite{Lieb1967SAM} employed it to characterize an electron trapped in its own hole as an approximation to the Hartree-Fock theory for a one component plasma.
For more physical background on equation  \eqref{eq1.2}, see \cite{Bahrami2014, Choquard-Stubbe-Vuffray2008DIE, Moroz-Van-Schaftingen2013JFA}  and the references therein. Mathematically,  the appearance of the  nonlocal term of $\Big(|x|^{-\alpha} *|u|^{p}\Big)|u|^{p-2}u$ introduces significant challenges in the analysis of the Choquard equation, which has attracted increasing attention in recent literature.  For $N=3$ and $\alpha=1$, Lieb \cite{Lieb1967SAM} shown the existence and uniqueness (modulo translations) of a minimizer to  \eqref{eq1.2} by using
symmetric decreasing rearrangement inequalities.  Subsequently, Lions in \cite{Lions1980NA} generalized the work of Lieb.
Ma and Zhao \cite{Ma-Zhao2010ARMA} proved that, for $p \geq 2$, all positive solutions must be radially symmetric and monotone decreasing about some fixed point by the method of moving planes.  Cingolani, Clapp and Secchi \cite{Cingolani-Clapp-Secchi2012ZAMP}
obtained some existence and multiplicity of nontrivial intertwining solution of  \eqref{eq1.2}  in the electromagnetic
case.  Moroz and Van
Schaftingen \cite{Moroz-VanSchaftingen2013JFA} verified the regularity, positivity, radial symmetry and decay asymptotics at infinity of the ground states of \eqref{eq1.2}. For further details and important advances on this subject, we refer the reader to \cite{Chen-Radulescu-Shu-Wei2025MA, Guo-Hu-Peng-Shuai2019CVPDE,Xia-Zhang2024SIAM} and the discussion in the sequel concerning the results of \eqref{eq1.2}.

When $s=0$ and $q=2$, equation \eqref{eq1.1} turns into the following  Hartree type Br\'{e}zis-Nirenberg problem
\begin{equation}\label{eq1.3}
\left\{\begin{array}{ll}
-\Delta u =
\Big( \displaystyle{\int_{\Omega}} \frac{|u(y)|^{2_{\alpha}^*} }{|x-y|^{\alpha}}dy\Big) |u|^{{2_{\alpha}^*}-1}+\lambda u,     \,  &x\in  \Omega, \\
u =0,    \, &  x\in \partial \Omega,
\end{array} \right.
\end{equation}
Gao and Yang in  \cite{Gao-Yang2018SCM} obtained some existence results in the spirit of the well-known paper by Brezis and Nirenberg \cite{Brezis-Nirenberg1983CPAM}.
Later, they \cite{Gao-Yang2017JMAA} also established some existence and multiplicity results for \eqref{eq1.3} when replacing linear perturbation by sublinear,  superlinear and nonlocal perturbations.
For $\alpha\in (2,4)$, $N\leq 4$ and $\lambda\in (0,\lambda_1)$,
Liu and Yang  \cite{Liu-Yang2025JMAA} shown that \eqref{eq1.3} admits an odd solution with exactly two nodal domain by the concentration compactness principle.
For $\alpha\in (0,4)$ and $N\leq 5$,
Yang, Ye and Zhao \cite{Yang-Ye-Zhao2023JDE} studied the existence and asymptotic behavior of the solutions of \eqref{eq1.3} by using the Lyapunov-Schmidt reduction procedure.
For $\alpha\in (0,4)$ and $N\leq 4$,
Yang and Zhao  \cite{Yang-Zhao2023JGA}  proved that the solution  of \eqref{eq1.3} blows up exactly at a critical point
of the Robin function that cannot be on the boundary  via the Lyapunov-Schmidt reduction method.
Later, Squassina, Yang and Zhao  \cite{Squassina-Yang-Zhao2023} studied the location of the blow-up points for single bubbling solutions of \eqref{eq1.3} provided that $\lambda>0$ is small enough by using the local Pohozaev identity and the blow-up analysis.
For $\alpha\in (0,4)$ and $N\leq 5$,
Pan, Wen and Yang \cite{Pan-Wen-Yang2025JDE} considered the qualitative analysis
the blow-up solutions of  \eqref{eq1.3}.
For more results about \eqref{eq1.3}, we refer to \cite{Chen-Wang2024CVPDE,He2022JMAA} and references therein.

An additional motivation for this work originates in recent studies on the positive solutions of the following equation
\begin{equation}\label{eq1.4}
\left\{\begin{array}{ll}
-\Delta u(x)=\lambda |u|^{p-1}+\frac{ |u|^{2^*(s)-1}}{|x|^s}, ~~ x\in  \Omega,
\\
u(x) =0,    \,   x\in \partial \Omega,
\end{array} \right.
\end{equation}
where $1<p<\frac{2N}{N-2}$ and $\Omega$ is a bounded domain in $\mathbb{R}^N$ with Lipschitz boundary. Equation \eqref{eq1.4} can be seen as the limiting equation of \eqref{eq1.1} as $\alpha\to N$. This is because the
nonlocal  term $\big(|x|^{-\alpha} * |u|^{p}\big) |u|^{p-2}u$ formally degenerates to the local term $|u|^{p-1}u$ as $\alpha\to N$.  For $0\in \partial\Omega$, $\lambda<0$ and $2^*(s)-1<p<2^*-1 $, the  problem of existence of positive solutions for \eqref{eq1.4} is  related to the Li-Lin's open problem \cite{Li-Lin2012ARMA}.
For $0\in \partial\Omega$, $N\geq4$ and  $\lambda>0$,  Ghoussoub and Kang \cite{Ghoussoub-Kang2004AIHP} obtained the existence of a positive solution of \eqref{eq1.4} provided that either $\frac{N}{N-2}<p< 2^*$ or $1<p< \frac{N}{N-2}$
and the mean curvature of $\partial\Omega$ at $0$ is negative.
When $0\in \partial\Omega$, $N\geq3$,  $\lambda<0$, $2<p<\frac{N}{N-2} $ and the mean curvature of $\partial\Omega$ at $0$ is negative, Hsia, Lin and Wadade \cite{Hsia-Lin-Wadade2010JFA} shown the
the existence of  a positive solution of  \eqref{eq1.4}. Moreover, they also considered the case of $\lambda>0$ and $p=2^*-1 $. For $0\in \Omega$, $N=3$, $\lambda>0$ and $2<p<2^*-1 $,
Cerami, Zhong, and Zou \cite{Cerami-Zhong-Zou2015CVPDE} verified that \eqref{eq1.4} possesses a positive solution if $p>3 $ or $p=3$ with $\lambda$ large enough.  For more results concerning \eqref{eq1.4}, we refer to the papers \cite{Hsia-Lin-Wadade2010JFA, Li-Lin2012ARMA,Zhong-Zou2016CCM}
and reference therein.

Motivated by the aforementioned work, in this paper we continue the study of the case $0\in \Omega$ and get the following existence results for \eqref{eq1.1}:

\begin{Thm}\label{Thm1.1}
Let \(\Omega\subset \mathbb{R}^N\) be a $C^1$ open bounded domain with $0\in \Omega$. Assume  $0\leq s\leq2$ and  $0<\alpha<N$, then  problem \eqref{eq1.1} has at least one nontrivial solution provided one of the following conditions hold:
\begin{itemize}
\item[$(1)$]  $\lambda, \mu>0$, $1<p<2_\alpha^*$ and  $2< q < 2^*(s)$,
\item[$(2)$]  $\lambda>0$, $0<\mu<\bar{\mu}=\frac{(N-2)^2}{4}$,  $1<p<2_\alpha^*$ and $2= q =2^*(s)$( that is, $s=2$),
\end{itemize}
\begin{itemize}
\item[$(3)$]  $\mu>0$,  $2< q =2^*(s)$ and one of the following conditions is satisfied:
    \begin{itemize}
    \item[(\romannumeral1)] $2_\alpha^*-1<p<2_\alpha^*$ and $\lambda>0$,
    \item[(\romannumeral2)] $1<p\leq2_\alpha^*-1$ with $\alpha\leq4$ and $\lambda>0$ is sufficiently large,
    \end{itemize}
\item[$(4)$]  $\lambda>0$, $p=2_\alpha^*$  and one of the following conditions is satisfied:
    \begin{itemize}
     \item[(\romannumeral1)] $N=3$, $s<1$, $2^*(s)-2<q< 2^*(s)$ and $\mu>0$,
     \item[(\romannumeral2)] $N=3$, $s<1$, $2< q\leq 2^*(s)-2$ and $\mu>0$ is sufficiently large,
     \item[(\romannumeral3)] $N=3$, $1\leq s<2$, $2<q< 2^*(s)$ and $\mu>0$,
     \item[(\romannumeral4)] $N\geq4$, $0<s<2$, $2<q< 2^*(s)$ and $\mu>0$.
    \end{itemize}
\end{itemize}
\end{Thm}

We conclude the introduction by mentioning that for the double critical case $p=2_{\alpha}^*$ and $2= q =2^*(s)$( that is, $s=2$),
\begin{equation}\label{eq1.5}
\left\{\begin{array}{ll}
-\Delta u(x)-\mu\frac{ u}{|x|^2}=
\Big( \displaystyle{\int_{\Omega}} \frac{|u(y)|^{2_{\alpha}^*} }{|x-y|^{\alpha}}dy\Big) |u(x)|^{2_{\alpha}^*-2}u(x),     \,  &x\in  \Omega,  \\
u(x) =0,    \, &  x\in \partial \Omega,
\end{array} \right.
\end{equation}
the least-energy solutions of are extremal functions for
$$S_{H,\alpha}(\Omega):=\inf_{u\in D_0^{1,2}(\Omega)\backslash \{0\} } \frac{\int_{\Omega}|\nabla u|^2dx -\mu\int_{\Omega}\frac{u^2}{|x|^2} dx
}{
\Big(\int_{\Omega}\int_{\Omega} \frac{|u|^{2_{\alpha}^*}(x)|u(y)|^{2_{\alpha}^*} }{|x-y|^{\alpha}}dxdy \Big)^{1/2_{\alpha}^*}}.$$
Guo and Tang \cite{Guo-Tang2025arXiv} shown the existence, symmetry and  asymptotic behavior of the extremal function of $S_{H,\alpha}(\mathbb{R}^N)$ via a suitable version of concentration-compactness property and the moving plane method.

In a forthcoming paper, we shall tackle the case $0\in\partial\Omega$, which is more in the spirit of Li-Lin's open problem. In this case, the study of nontrivial solutions was initiated by Ghoussoub and Kang \cite{Ghoussoub-Kang2004AIHP} and developed by Ghoussoub and Robert \cite{Ghoussoub-Robert2006GFA}.
Although we are concerned with the case $0\in\Omega$ in this paper, we face here new additional difficulties due to the combined effect of the nonlocal term $\Big(|x|^{-\alpha} *|u|^{p}\Big)|u|^{p-2}u$  and the Hardy-Sobolev critical term.
In order to obtain the existence of nontrivial solution, a deep analysis is needed to establish a compactness property and run the variational method.

\medskip

The paper is organized as follows. In section \ref{sec:prelim} we collect some preliminary results and then in section \ref{sec:proof} we give the proof of the main result.

\par
\section{Preliminary results} \label{sec:prelim}
Before proving our results we introduce some notations and collect some basic results. Let us consider the space $H^1_0(\Omega)$ which is the completion of $C_0^{\infty}(\Omega)$ with the norm
$$||u||:=\left(\int_{\Omega} |\nabla u |^2 dx \right)^{\frac{1}{2}}.$$
As a Hilbert space, it is endowed with inner product $\langle u,v \rangle := \left(\int_{\Omega} \nabla u \cdot \nabla v dx \right)^{\frac{1}{2}}$.
It is well-known that $H^1_0(\Omega)\hookrightarrow L^{p}(\Omega, |x|^{-s}dx)$ continuously for $p\in [1, 2^*(s)]$, compactly for $p\in [1, 2^*(s))$ with $0\leq s<2$.

For the sake of  convenience we recall the well-known Hardy-Littlewood-Sobolev inequality \cite{Lieb-Loss2001book}:
\begin{Pro}\label{Pro2.3}
(Hardy-Littlewood-Sobolev inequality) Let $t,r>1$, and $0<\alpha< N$ with $1/t+ 1/r+ \alpha/N=2 ,~f\in L^t(\R^N)$  and $g\in L^r(\R^N)$. There exists a sharp constant  $C(t,r,N,\mu)>0$, independent of $f$ and $g$, such that
\begin{equation}\label{eq2.1}
\int_{\mathbb{R}^N}\int_{\R^N}\frac{f(x)g(y)}{|x-y|^{\alpha}}dxdy \leq C(t,r,N,\alpha)||f||_t||g||_r.
\end{equation}
\end{Pro}

\medskip

Moreover, if $t=r=\frac{2N}{2N-\alpha}$, then
$$C(t,r,N,\alpha)=C(N,\alpha)= \pi^{\frac{\alpha}{2}} \frac{\Gamma(\frac{N-\alpha}{2})} {\Gamma(\frac{2N-\alpha}{2})} \bigg(\frac{\Gamma(\frac{N}{2})} {\Gamma(N)} \bigg)^{\frac{\alpha}{N}-1}.$$
In this case, the equality in \eqref{eq2.1} holds if and only if $f\equiv C g$ and
$$g(x)=A\Big(a+|x-x_0|^2\Big)^{\frac{\alpha-2N}{2}}$$
for some $A\in \mathbb{C}, x_0\in \mathbb{R}^N$ and $0\neq a \in \mathbb{R}$.

Let us also recall that if $f(x)=h(x)=u^p(x)$ in  \eqref{eq2.1},  the following integral
$$\int_{\R^N}\int_{\R^N}\frac{|u(x)|^p|u(y)|^p
}{|x-y|^{\alpha}}dxdy $$
is well-defined provided
$$\frac{2N-\alpha}{N}\leq p \leq \frac{2N-\alpha}{N-2}. $$ $2_{*\alpha}=\frac{2N-\alpha}{N}$ is called the lower critical  exponent and $2_{\alpha}^*=\frac{2N-\alpha}{N-2}$  the upper critical exponent in the sense of the Hardy-Littlewood-Sobolev inequality.

\par
\section{Proof of the main result} \label{sec:proof}
It is well known that the nontrivial solutions of \eqref{eq1.1} are equivalent to nonzero critical points of the following energy functional
\begin{equation*}
I(u) := \frac{1}{2} \int_{\Omega} |\nabla u|^2dx -\frac{\lambda}{2p} \int_{\Omega}\int_{\Omega} \frac{|u(x)|^{p}|u(y)|^{p} }{|x-y|^{\alpha}}dxdy-  \frac{\mu}{q} \int_{\Omega} \frac{|u|^{q}}{|x|^s}dx .
\end{equation*}

\vspace{3mm}
It is easy to check that the functional $\mathcal{I}$ has a the Mountain-Pass geometry.
\begin{Lem}\label{Lem2.1}
It holds:
\begin{itemize}
\item[$(1)$]  There exist $\beta, \rho>0$ such that $\mathcal{I}(u)\geq \beta$ whenever $||u||=\rho$.
\item[$(2)$]  There is an $e\in H_0^1(\Omega)$ with $||e||\geq \rho$ such that $\mathcal{I}(e)<0$.
\end{itemize}
\end{Lem}
\begin{proof}
$(1)$ If $p\in(1, 2_{\alpha}^*]$ and $q\in(2, 2^*(s)]$, then  by the Sobolev embedding and the Hardy-Littlewood-Sobolev inequality, for all
$u\in H_0^1(\Omega)\setminus\{0\}$, we have that
$$\begin{aligned}
\mathcal{I}(u)&= \frac{1}{2} \int_{\Omega} |\nabla u|^2dx -\frac{\lambda}{2p} \int_{\Omega}\int_{\Omega} \frac{|u(x)|^{p}|u(y)|^{p} }{|x-y|^{\alpha}}dxdy -  \frac{\mu}{q} \int_{\Omega} \frac{|u|^{q}}{|x|^s}dx\\
&\geq \frac{1}{2} ||u||^2 -\frac{\lambda}{2p} C ||u||^{2p}-  \frac{\mu}{q}C ||u||^{q},
\end{aligned}$$
which means that we can choose some $\alpha,\rho>0$ such that $\mathcal{I}(u)\geqslant\alpha$ for $\|u\|=\rho$.

If $0<\mu<\bar{\mu}$, $1<p<2_\alpha^*$ and $2= q =2^*(s)$( that is, $s=2$), by using the Hardy-Littlewood-Sobolev inequality again,
$$\begin{aligned}
\mathcal{I}(u)&= \frac{1}{2} \int_{\Omega} |\nabla u|^2dx  -  \frac{\mu}{2} \int_{\Omega} | u|^2dx -\frac{\lambda}{2p} \int_{\Omega}\int_{\Omega} \frac{|u(x)|^{p}|u(y)|^{p} }{|x-y|^{\alpha}}dxdy\\
&\geqslant \frac{\bar{\mu}-\mu}{\bar{2\mu}}\|u\|^2-\frac{\lambda}{2p} C ||u||^{2p},  \\
\end{aligned}$$
which gives  that  we can choose some $\alpha,\rho>0$ such that $\mathcal{I}(u)\geqslant\alpha$ for $\|u\|=\rho.$

$(2)$  For some $u_0\in H_0^1(\Omega)\setminus\{0\}$, we have
$$\begin{aligned}
\mathcal{I}(tu_0)&=\frac{t^2}{2}\int_\Omega|\nabla u_0|^2dx -\frac{t^{2p}\lambda}{2p} \int_{\Omega}\int_{\Omega} \frac{|u(x)|^{p}|u(y)|^{p} }{|x-y|^{\alpha}}dxdy -  \frac{t^q\mu}{q} \int_{\Omega} \frac{|u|^{q}}{|x|^s}dx
\end{aligned}$$
for $t > $ large enough. Then we may take an $e:= t_*u_0$ for some $t_*0 > 0 $ and  $(2)$  follows.
\end{proof}

\medskip

Concerning the Palais-Smale condition we have the following:
\begin{Lem}\label{Lem2.2}
Let $0\leq s\leq2$, $0<\alpha<N$, $p$ and $q$ satisfy $p\in(1, 2_{\alpha}^*]$, $q\in[2, 2^*(s)]$. Then
\begin{itemize}
\item[$(1)$] If $1<p<2_\alpha^*$ and  $2< q < 2^*(s)$, then  for any $\lambda, \mu>0$, $\mathcal{I}$ satisfies $(PS)_c$ for all $c\in \mathbb{R}$.
\item[$(2)$]  If $1<p<2_\alpha^*$ and $2< q =2^*(s)$, then  for any $\lambda, \mu>0$, $\mathcal{I}$ satisfies $(PS)_c$ for all
    \[c<\frac{2-s}{2(N-s)} \frac{1}{\mu^{\frac{2}{2^*(s)-2}}} \mu_{s}(\mathbb{R}^N)^{\frac{N-s}{2-s}}.\]
\item[$(3)$] If $p=2_\alpha^*$ and  $2< q <2^*(s)$  for any $\lambda, \mu>0$, $\mathcal{I}$ satisfies $(PS)_c$ for all
    \[c< \frac{N-\alpha+2}{2(2N-\alpha)}  \Big(\frac{1}{\lambda C(N,\alpha) }  \Big)^{\frac{1}{2_\alpha^*-1}} S^{\frac{2_\alpha^*}{2_\alpha^*-1}}.\]
\item[$(4)$] If $1<p<2_\alpha^*$ and $2= q =2^*(s)$( that is, $s=2$), for any $\lambda>0$, $0<\mu<\bar{\mu}$, $\mathcal{I}$ satisfies $(PS)_c$ for all $c\in \mathbb{R}$.
\end{itemize}
\end{Lem}
\begin{proof}
Suppose that ${u_n}$ is a $(PS)_c$ sequence of $\mathcal{I}$, that is,
\begin{equation*}\label{eq}
\mathcal{I}(u_n)\to c~~\text{ and}~~~\mathcal{I}'(u_n)\to 0~\text{in }~ (H_0^1(\Omega))^{-1}.
\end{equation*}
We obtain
\begin{equation}\label{2.17eq}
\frac{1}{2} \int_{\Omega} |\nabla u_n|^2dx -\frac{\lambda}{2p} \int_{\Omega}\int_{\Omega} \frac{|u_n(x)|^{p}|u_n(y)|^{p} }{|x-y|^{\alpha}}dxdy-  \frac{\mu}{q} \int_{\Omega} \frac{|u_n|^{q}}{|x|^s}dx=\mathcal{I}(u_n) +o_n(1)
\end{equation}
and
$$\int_{\Omega}  |\nabla u_n|^2dx - \lambda\int_{\Omega}\int_{\Omega} \frac{|u_n(x)|^{p}|u_n(y)|^{p} }{|x-y|^{\alpha}}dxdy -\mu\int_{\Omega} \frac{|u_n|^{q}}{|x|^s}dx=o_n(1)||u_n||,$$
as $n\to +\infty$. Then  we obtain
\begin{equation}\label{2.18eq}
\begin{aligned}
c+o_n(1)||u_n||&\geq \mathcal{I}(u_n) -\frac{1}{2}\langle \mathcal{I}'(u_n), u_n\rangle\\
&=\Big(\frac{1}{2}-\frac{1}{2p} \Big)\lambda \int_{\Omega}\int_{\Omega} \frac{|u_n(x)|^{p}|u_n(y)|^{p} }{|x-y|^{\alpha}}dxdy+ \Big(\frac{1}{2}-\frac{1}{q} \Big) \mu \int_{\Omega} \frac{|u_n|^{q}}{|x|^s}dx.
\end{aligned}
\end{equation}
We first show that $\{u_n\}$ is bounded in $H_0^1(\Omega)$ by splitting the proof in two cases.

Case $1$: $s<2$.

In this case, we deduce that $q>2$, then Combining this  with \eqref{2.18eq} and $p>1$, we have
$$ \int_{\Omega}\int_{\Omega} \frac{|u_n(x)|^{p}|u_n(y)|^{p} }{|x-y|^{\alpha}}dxdy\leq C (1+||u_n||)~\text{and}~  \int_{\Omega} \frac{|u_n|^{q}}{|x|^s}dx \leq C (1+||u_n||),$$
which, combined with \eqref{2.17eq}, implies
\begin{equation}\label{2.19eq}
\begin{aligned}
||u_n||^2&=2C+o_n(1)+\frac{\lambda}{p} \int_{\Omega}\int_{\Omega} \frac{|u_n(x)|^{p}|u_n(y)|^{p} }{|x-y|^{\alpha}}dxdy+  \frac{2\mu}{q} \int_{\Omega} \frac{|u_n|^{q}}{|x|^s}dx\\
&\leq C (1+||u_n||).
\end{aligned}
\end{equation}
Then the boundedness of $\{u_n\}$ in $H_0^1(\Omega)$ can be readily deduced.

\medskip

Case $2$: $s=2$.

In this case we get $q=2$, and then we deduce from \eqref{2.18eq} and $p>1$  that
$$ \int_{\Omega}\int_{\Omega} \frac{|u_n(x)|^{p}|u_n(y)|^{p} }{|x-y|^{\alpha}}dxdy\leq C (1+||u_n||).$$
So the  Hardy's inequality, $0<\mu < \bar{\mu}$ and \eqref{2.17eq} give that
\begin{equation*}\label{2.19eq*}
\begin{aligned}
\Big(1-\frac{\mu}{\bar{\mu}} \Big)||u_n||^2&=2C+o_n(1)+\frac{2\lambda}{p} \int_{\Omega}\int_{\Omega} \frac{|u_n(x)|^{p}|u_n(y)|^{p} }{|x-y|^{\alpha}}dxdy\\
&\leq C (1+||u_n||),
\end{aligned}
\end{equation*}
which implies the boundedness of $\{u_n\}$.

\medskip

Thus, up to a subsequence, there exists $u_0\in H_0^1(\Omega)$ such that $u_n\rightharpoonup u_0$ in $H_0^1(\Omega)$,  $u_n\rightharpoonup  u_0$ in $L^{2^*(s)}(\Omega)$ and  $u_n\to u_0$ a.e. in $\Omega$ as $n\to +\infty$. Thus, If $1<p<2_\alpha^*$ and  $2< q < 2^*(s)$, we have that
$$\int_{\Omega}\int_{\Omega} \frac{|u_n(x)|^{p}|u_n(y)|^{p} }{|x-y|^{\alpha}}dxdy \to \int_{\Omega}\int_{\Omega} \frac{|u_0(x)|^{p}|u_0(y)|^{p} }{|x-y|^{\alpha}}dxdy$$
and
$$\int_{\Omega} \frac{|u_n|^{q}}{|x|^s}dx \to \int_{\Omega} \frac{|u_0|^{q}}{|x|^s}dx $$
as $n\to +\infty$.
If $p=2_\alpha^*$ and  $ q =2^*(s)$, we have that
$$\int_{\Omega}\frac{|u_n(y)|^{2_\alpha^*}}{|x-y|^{\alpha}}dy |u_n|^{2_{\alpha}^*-2}u_n\rightharpoonup \int_{\Omega}\frac{|u_0(y)|^{2_\alpha^*}}{|x-y|^{\alpha}}dy |u_0|^{2_{\alpha}^*-2}u_0~~\text{in}~L^{\frac{2N}{N+2}}(\Omega)$$
and
$$u_n\rightharpoonup  u_0 ~~\text{in}~L^{2^*(s)}(\Omega, |x|^{-s}dx)$$
as $n\to +\infty$.  Since $\mathcal{I}'(u_n)\to 0$ in  $(H_0^1(\Omega))^{-1}$, for any $\phi\in H_0^1(\Omega)$ we have
$$o_n(1)= \int_{\Omega} \nabla u_n \nabla \phi dx
-\lambda\int_{\Omega} \int_{\Omega}\frac{|u_n(y)|^{p} |u_n(x)|^{p-2}u_n(x)\phi(x) }{|x-y|^{\alpha}}dxdy-\mu \int_{\Omega} \frac{|u_n|^{q-2}u_n \phi}{|x|^s} dx .$$
Thus
$$\int_{\Omega} \nabla u_0 \nabla \phi dx  =\lambda\int_{\Omega} \int_{\Omega}\frac{|u_0(y)|^{p} |u_0(x)|^{p-2}u_0(x)\phi(x) }{|x-y|^{\alpha}}dxdy+\mu \int_{\Omega} \frac{|u_0|^{q-2}u_0 \phi}{|x|^s} dx,~\forall \phi \in H_0^1(\Omega),$$
which yields that $u_0$ is a weak solution of \eqref{eq1.1}. Moreover,
$$\mathcal{I}(u_0)=\mathcal{I}(u_0) -\frac{1}{2}\langle \mathcal{I}'(u_0), u_0\rangle=\Big(\frac{1}{2}-\frac{1}{2p} \Big)\lambda \int_{\Omega}\int_{\Omega} \frac{|u_0(x)|^{p}|u_0(y)|^{p} }{|x-y|^{\alpha}}dxdy+ \Big(\frac{1}{2}-\frac{1}{q} \Big) \mu \int_{\Omega} \frac{|u_0|^{q}}{|x|^s}dx\geq0.$$

Next we prove that $u_n\to u_0$ in $H_0^1(\Omega)$.   Denote $v_n:=u_n- u_0$, then $v_n\rightharpoonup  0$ in $L^{2^*}(\Omega)$, $v_n\to 0$ in $L^q(\Omega)$ for $q\in [1, 2^*)$ and  $v_n\to 0$ a.e. in $\Omega$ as $n\to +\infty$. We consider each case separately.

$(1)$ Let  $1<p<2_\alpha^*$ and  $2< q < 2^*(s)$.
By the Hardy-Littlewood-Sobolev inequality, we know that
$$\int_{\Omega}\int_{\Omega} \frac{|v_n(x)|^{p}|v_n(y)|^{p} }{|x-y|^{\alpha}}dxdy \leq C(N,\alpha) \Big(\int_{\Omega} |v_n|^{\frac{2N}{2N-\alpha}p}\Big)^{\frac{2N-\alpha}{2N}}\rightarrow 0.$$
combining this with $v_n\rightharpoonup  0 ~~\text{in}~L^{q}(\Omega, |x|^{-s}dx)$ and $\mathcal{I}'(u_n)\to 0$ in  $(H_0^1(\Omega))^{-1}$, we deduce that
\begin{equation}\label{eq2.201}
\begin{aligned}
o_n(1)&=\langle \mathcal{I}'(u_n), u_n\rangle-\langle \mathcal{I}'(u_0), u_0\rangle \\
&=||v_n||^2-\lambda \int_{\Omega}\int_{\Omega} \frac{|v_n(x)|^{p}|v_n(y))|^{p} }{|x-y|^{\alpha}}dxdy -\mu \int_{\Omega} \frac{|v_n|^{q}}{|x|^s}dx,
\end{aligned}
\end{equation}
which shows that $u_n\to u_0$ in $H_0^1(\Omega)$ if  $1<p<2_\alpha^*$ and  $2< q < 2^*(s)$.

\medskip

$(2)$ Let  $1<p<2_\alpha^*$, $2< q =2^*(s)$ and \(c<\frac{2-s}{2(N-s)} \frac{1}{\mu^{\frac{2}{2^*(s)-2}}} \mu_{s}(\mathbb{R}^N)^{\frac{N-s}{2-s}}\). By the Brezis-Lieb Lemma, we have
$$\int_{\Omega} |\nabla u_n|^2dx=\int_{\Omega} |\nabla v_n|^2dx + \int_{\Omega} |\nabla u_0|^2dx+o_n(1)$$
and
$$ \int_{\Omega} \frac{|u_n|^{2^*(s)}}{|x|^s}dx =\int_{\Omega} \frac{|v_n|^{2^*(s)}}{|x|^s}dx+\int_{\Omega} \frac{|u_0|^{2^*(s)}}{|x|^s}dx+o_n(1)$$
as $n\to +\infty$. It follows from $\mathcal{I}(u_n)\to c$ and $\mathcal{I}'(u_n)\to 0$ in $(H_0^1(\Omega))^{-1}$ that
\begin{equation}\label{2.21eq}
\begin{aligned}
c+o_n(1)= \mathcal{I}(u_n)
& = \frac{1}{2} \int_{\Omega} |\nabla v_n|^2dx +\frac{1}{2} \int_{\Omega} |\nabla u_0|^2dx  -\frac{\lambda}{2p} \int_{\Omega}\int_{\Omega} \frac{|u_0(x)|^{p}|u_0(y)|^{p} }{|x-y|^{\alpha}}dxdy\\
&- \frac{\mu}{2^*(s)}\int_{\Omega} \frac{|v_n|^{2^*(s)}}{|x|^s}dx- \frac{\mu}{2^*(s)}\int_{\Omega} \frac{|u_0|^{2^*(s)}}{|x|^s}dx+o_n(1)\\
&=\mathcal{I}(u_0) +\frac{1}{2} \int_{\Omega} |\nabla v_n|^2dx - \frac{\mu}{2^*(s)}\int_{\Omega} \frac{|v_n|^{2^*(s)}}{|x|^s}dx+o_n(1)
\end{aligned}
\end{equation}
and
\begin{equation}\label{2.22eq}
\begin{aligned}
o_n(1)
&=\int_{\Omega} |\nabla v_n|^2dx +\int_{\Omega} |\nabla u_0|^2dx -\lambda\int_{\Omega}\int_{\Omega} \frac{|u_0(x)|^{p}|u_0(y)|^{p} }{|x-y|^{\alpha}}dxdy \\
&  -\mu\int_{\Omega} \frac{|v_n|^{2^*(s)}}{|x|^s}dx-\mu\int_{\Omega} \frac{|u_0|^{2^*(s)}}{|x|^s}dx \\
&=\langle \mathcal{I}'(u_0), u_0\rangle +\int_{\Omega} |\nabla v_n|^2dx -\mu\int_{\Omega} \frac{|v_n|^{2^*(s)}}{|x|^s}dx\\
&= \int_{\Omega} |\nabla v_n|^2dx -\mu\int_{\Omega} \frac{|v_n|^{2^*(s)}}{|x|^s}dx.
\end{aligned}
\end{equation}
Thus by using \eqref{2.21eq} and \eqref{2.22eq}, we get
\begin{equation}\label{2.221eq}
\begin{aligned}
\Big( \frac{1}{2} - \frac{1}{2^*(s)}\Big)\int_{\Omega} |\nabla v_n|^2dx &=\frac{1}{2} \int_{\Omega} |\nabla v_n|^2dx - \frac{\mu}{2^*(s)}\int_{\Omega} \frac{|v_n|^{2^*(s)}}{|x|^s}dx\\
&\leq c+o_n(1)\\
&< \Big( \frac{1}{2} - \frac{1}{2^*(s)}\Big) \frac{1}{\mu^{\frac{2}{2^*(s)-2}}} \mu_{s}(\mathbb{R}^N)^{\frac{2^*(s)}{2^*(s)-2}}.
\end{aligned}
\end{equation}
If $\int_{\Omega} |\nabla v_n|^2dx\to 0$ as $n\to +\infty$, the proof is completed. Otherwise, by the definition of \(\mu_{s}(\mathbb{R}^N)\) and \eqref{2.22eq} we obtain that
\[ c\geq \Big( \frac{1}{2} - \frac{1}{2^*(s)}\Big) \int_{\Omega}  |\nabla u_n|^2dx  \geq \Big( \frac{1}{2} - \frac{1}{2^*(s)}\Big) \frac{1}{\mu^{\frac{2}{2^*(s)-2}}} \mu_{s}(\mathbb{R}^N)^{\frac{2^*(s)}{2^*(s)-2}}=\frac{2-s}{2(N-s)} \frac{1}{\mu^{\frac{2}{2^*(s)-2}}} \mu_{s}(\mathbb{R}^N)^{\frac{N-s}{2-s}},\]
which  is a contradiction to \eqref{2.221eq}. Thus  $u_n\to u_0$ in $H_0^1(\Omega)$.

\medskip

$(3)$ Let  $p=2_\alpha^*$, $2< q <2^*(s)$ and \(c< \frac{N-\alpha+2}{2(2N-\alpha)}  \Big(\frac{1}{\lambda C(N,\alpha) }  \Big)^{\frac{1}{2_\alpha^*-1}} S^{\frac{2_\alpha^*}{2_\alpha^*-1}}\). It follows from \cite{Gao-Yang2018SCM} that
\( S =C(N, \alpha)^{1/2_\alpha^*} S_{H,L} \), where
\begin{equation*}
S_{H,L}:=\inf_{u\in D^{1,2}(\mathbb{R}^N)\backslash \{0\} } \frac{\int_{\mathbb{R}^N}|\nabla u|^2dx
}{
\Big(\int_{\mathbb{R}^N}\int_{\mathbb{R}^N} \frac{|u|^{2_{\alpha}^*}(x)|u(y)|^{2_{\alpha}^*} }{|x-y|^{\alpha}}dxdy \Big)^{1/2_{\alpha}^*}}.
\end{equation*}
Thus, we get that
\[c<\frac{N-\alpha+2}{2(2N-\alpha)}  \Big(\frac{1}{\lambda C(N,\alpha) }  \Big)^{\frac{1}{2_\alpha^*-1}} S^{\frac{2_\alpha^*}{2_\alpha^*-1}} =\frac{N+2-\alpha}{2(2N-\alpha)} \frac{1}{\lambda^{\frac{1}{2_\alpha^*-1}}} S_{H,L}^{\frac{2N-\alpha}{N+2-\alpha}} .\]
By the Brezis-Lieb Lemma, we have
$$\int_{\Omega}\int_{\Omega}\frac{|u_n(y)|^{2_\alpha^*} |u_n(x)|^{2_{\alpha}^*} }{|x-y|^{\alpha}}dxdy=\int_{\Omega}\int_{\Omega}\frac{|v_n(y)|^{2_\alpha^*} |v_n(x)|^{2_{\alpha}^*} }{|x-y|^{\alpha}}dxdy +\int_{\Omega}\int_{\Omega}\frac{|u_0(y)|^{2_\alpha^*} |u_0(x)|^{2_{\alpha}^*} }{|x-y|^{\alpha}}dxdy +o_n(1).$$
Similar to \eqref{2.21eq} and \eqref{2.22eq}, we have that
\begin{equation}\label{2.21eq2}
\begin{aligned}
c\leftarrow \mathcal{I}(u_n) =\mathcal{I}(u_0) +\frac{1}{2} \int_{\Omega} |\nabla v_n|^2dx - \frac{\lambda}{2\cdot2_\alpha^*}\int_{\Omega}\int_{\Omega}\frac{|v_n(y)|^{2_\alpha^*} |v_n(x)|^{2_{\alpha}^*} }{|x-y|^{\alpha}}dxdy
\end{aligned}
\end{equation}
and
\begin{equation}\label{2.22eq2}
\begin{aligned}
o_n(1)= \int_{\Omega} |\nabla v_n|^2dx -\lambda\int_{\Omega}\int_{\Omega}\frac{|v_n(y)|^{2_\alpha^*} |v_n(x)|^{2_{\alpha}^*} }{|x-y|^{\alpha}}dxdy .
\end{aligned}
\end{equation}
Thus by using \eqref{2.21eq2} and \eqref{2.22eq2}, we get

\begin{equation}\label{2.221eq2}
\begin{aligned}
\Big( \frac{1}{2} - \frac{1}{2\cdot2_\alpha^*}\Big)\int_{\Omega} |\nabla v_n|^2dx &=\frac{1}{2} \int_{\Omega} |\nabla v_n|^2dx - \frac{\lambda}{2\cdot2_\alpha^*}\int_{\Omega}\int_{\Omega}\frac{|v_n(y)|^{2_\alpha^*} |v_n(x)|^{2_{\alpha}^*} }{|x-y|^{\alpha}}dxdy\\
&\leq c+o_n(1)\\
&< \Big( \frac{1}{2} - \frac{1}{2\cdot2_\alpha^*}\Big) \frac{1}{\lambda^{\frac{1}{2_\alpha^*-1}}} S_{H,L}^{\frac{2_\alpha^*}{2_\alpha^*-1}}=\frac{N+2-\alpha}{2(2N-\alpha)} \frac{1}{\lambda^{\frac{1}{2_\alpha^*-1}}} S_{H,L}^{\frac{2N-\alpha}{N+2-\alpha}} .
\end{aligned}
\end{equation}
If $\int_{\Omega} |\nabla v_n|^2dx\to 0$ as $n\to +\infty$, the proof is completed. Otherwise, by the definition of \(S_{H,L}\), \eqref{2.21eq2} and \eqref{2.22eq2} we obtain that
\[ c\geq \Big( \frac{1}{2} - \frac{1}{2\cdot2_\alpha^*}\Big) \int_{\Omega}  |\nabla u_n|^2dx  \geq \Big( \frac{1}{2} - \frac{1}{2\cdot2_\alpha^*}\Big) \frac{1}{\lambda^{\frac{1}{2_\alpha^*-1}}} S_{H,L}^{\frac{2_\alpha^*}{2_\alpha^*-1}}=\frac{N+2-\alpha}{2(2N-\alpha)} \frac{1}{\lambda^{\frac{1}{2_\alpha^*-1}}} S_{H,L}^{\frac{2N-\alpha}{N+2-\alpha}} ,\]
which  is a contradiction to \eqref{2.221eq2}. Thus  $u_n\to u_0$ in $H_0^1(\Omega)$.

\medskip

$(4)$ Let  $1<p<2_\alpha^*$, $2= q =2^*(s)$( that is, $s=2$). By the Brezis-Lieb Lemma, we have
$$\int_{\Omega} \frac{|u_n|^2}{|x|^2}dx=\int_{\Omega}\frac{|v_n|^2}{|x|^2}dx + \int_{\Omega} \frac{|u_0|^2}{|x|^2}dx+o_n(1).$$
Thus, we have
\begin{equation}\label{2.22eq3}
\begin{aligned}
o_n(1)
&=\langle \mathcal{I}'(u_0), u_0\rangle +\int_{\Omega} |\nabla v_n|^2dx -\mu\int_{\Omega} \frac{|v_n|^{2}}{|x|^2}dx\\
&= \int_{\Omega} |\nabla v_n|^2dx -\mu\int_{\Omega} \frac{|v_n|^{2}}{|x|^2}dx\\
&\geq \Big(1-\frac{\mu}{\bar{\mu}}\Big)\int_{\Omega} |\nabla v_n|^2dx\geq0,
\end{aligned}
\end{equation}
which implies  $u_n\to u_0$ in $H_0^1(\Omega)$ provided that $0<\mu<\bar{\mu}$.
\end{proof}

\begin{Lem}\label{Lem2.3}
Let \(\Omega\subset \mathbb{R}^N\) be an  open bounded domain with $0\in \Omega$. Assume  $0\leq s\leq2$, $0<\alpha<N$, $p\in(1, 2_{\alpha}^*]$ and $q\in[2, 2^*(s)]$. Then
\begin{itemize}
\item[$(1)$]  if $q=2^*(s)$, $\mu>0$ and one of the following conditions is satisfied:
    \begin{itemize}
    \item[(\romannumeral1)] $2_\alpha^*-1<p<2_\alpha^*$ and $\lambda>0$,
    \item[(\romannumeral2)] $1<p\leq2_\alpha^*-1$ with $\alpha\leq4$ and $\lambda>0$ is sufficiently large,
    \end{itemize}
    then there exists a nonnegative function $u_{s,\varepsilon}\in H_0^1(\Omega)$ such that
    \[\max_{ t>0}\mathcal{I}(tu_{s,\varepsilon})<\frac{2-s}{2(N-s)} \frac{1}{\mu^{\frac{2}{2^*(s)-2}}} \mu_{s}(\mathbb{R}^N)^{\frac{N-s}{2-s}}.\]
    where $\mu_{s}(\mathbb{R}^N)$ is the best Sobolev-Hardy constant.
\item[$(2)$] if $p=2_\alpha^*$, $\lambda>0$ and one of the following conditions id satisfied:
    \begin{itemize}
     \item[(\romannumeral1)] $N=3$, $s<1$, $\frac{2(N-s)}{N-2}-2<q< \frac{2(N-s)}{N-2}$ and $\mu>0$,
     \item[(\romannumeral2)] $N=3$, $s<1$, $2< q\leq \frac{2(N-s)}{N-2}-2$ and $\mu>0$ is sufficiently large,
     \item[(\romannumeral3)] $N=3$, $1\leq s<2$, $2<q< \frac{2(N-s)}{N-2}$ and $\mu>0$,
     \item[(\romannumeral4)] $N\geq4$, $0<s<2$, $2<q< \frac{2(2-s)}{N-2}$ and $\mu>0$,
    \end{itemize}
  then there exists a nonnegative function $u_{\varepsilon}\in H_0^1(\Omega)$ such that
    \[\max_{ t>0}\mathcal{I}(tu_{\varepsilon})\leq \frac{N-\alpha+2}{2(2N-\alpha)}  \Big(\frac{1}{\lambda C(N,\alpha) }  \Big)^{\frac{1}{2_\alpha^*-1}} S^{\frac{2_\alpha^*}{2_\alpha^*-1}},\]
    where $S$ is the best Sobolev constant.
\end{itemize}
\end{Lem}
\begin{proof}
$(1)$ Let $1<p<2_\alpha^*$ and $2< q =2^*(s)$.
We start recalling that in \cite{Lieb1983AM,Talenti1976AMPA} for some $k>0$ and \( 0 \leq s < 2 \), the following function
\[ U_s(x) := \big(k\cdot(N-s)(N-2)  \big)^{\frac{N-2}{2(2-s)}} \left( k + |x|^{2-s} \right)^{-\frac{N-2}{2-s}} \]
is a least-energy solution to the  elliptic equation:
\[\begin{cases}
\displaystyle
\Delta u + \frac{u^{2^*(s)-1}}{|x|^{s}} = 0 & \text{in } \mathbb{R}^N, \\
u > 0 & \text{in } \mathbb{R}^N.
\end{cases}\]
Moreover, the best Hardy-Sobolev constant
\[ \mu_{s} \left( \mathbb{R}^N \right):=\inf \biggl\{ \int_{\mathbb{R}^N} |\nabla u|^{2} \, dx \, \biggl| \,  u \in D^{1,2}(\mathbb{R}^N) \text{ and } \int_{\mathbb{R}^N} \frac{|u|^{2^{*}(s)}}{|x|^{s}} \, dx=1 \biggr\}\]
can be achieved by $U_s(x)$.
Since \( 0 \in \Omega \), we can choose \( \rho > 0 \) such that \( B(0, 2\rho) \subset \Omega \) and then define a nonnegative function \( \psi(x) \in \mathcal{D}(\Omega) \) such that \( \psi(x) \equiv 1 \) on \( B(0, \rho) \), \( \psi(x) = 0 \) if \( |x| > 2\rho \). Set, for \( \varepsilon > 0 \),

\[U_{s,\varepsilon}(x) := \varepsilon^{-\frac{N-2}{2}} U_s \left( \frac{x}{\varepsilon} \right)  \quad \text{and} \quad u_{s,\varepsilon}(x) := \psi(x) U_{s,\varepsilon}(x).\]
By the definition of \( U_s \) and scaling properties, we have
\[\int_{\mathbb{R}^N} |\nabla U_{s,\varepsilon}|^2 \, dx=\int_{\mathbb{R}^N} \frac{|U_{s,\varepsilon}|^{2^{\ast}(s)}}{|x|^{s}} \, dx = \mu_{s} (\mathbb{R}^{N})^{\frac{N-s}{2-s}}.\]
By direct calculation we obtain, as \(\varepsilon \to 0^+\),
\[\int_{\Omega} |\nabla u_{s,\varepsilon}|^{2} \, dx = \int_{\mathbb{R}^{N}} |\nabla U_{s,\varepsilon}|^{2} \, dx + O(\varepsilon^{N-2}) = \mu_{s} (\mathbb{R}^{N})^{\frac{N-s}{2-s}} + O(\varepsilon^{N-2}),\]
\[\int_{\Omega} \frac{|u_{s,\varepsilon}|^{2^{*}(s)}}{|x|^{s}} \, dx = \int_{\mathbb{R}^{N}} \frac{|U_{s,\varepsilon}|^{2^{*}(s)}}{|x|^{s}} \, dx + O(\varepsilon^{N-s}) = \mu_{s} (\mathbb{R}^{N})^{\frac{N-s}{2-s}} + O(\varepsilon^{N-s}),\]
and
$$\begin{aligned}
&\int_\Omega\int_\Omega \frac{|u_{s,\varepsilon}(x)|^{p} |u_{s,\varepsilon}(y)|^{p}}{|x-y|^\alpha}dxdy\\
&\geqslant\int_{B_\delta} \int_{B_\delta}\frac{|U_{s,\varepsilon}(x) |^{p}|U_{s,\varepsilon}(y)|^{p}}{|x-y|^\alpha}dxdy \\
&=\int_{B_\delta}\int_{B_\delta}\frac{C\varepsilon^{(2-N)p}}{ \big(|k +|\frac{x}{\varepsilon}|^{2-s} \big)^{\frac{(N-2)p}{2-s}}|x-y|^\alpha \big(|k +|\frac{y}{\varepsilon}|^{2-s} \big)^{\frac{(N-2)p}{2-s}}}dxdy\\
&= \int_{B_\delta}\int_{B_\delta}\frac{C\varepsilon^{2N-\alpha-(N-2)p}}{ \big(k +|\frac{x}{\varepsilon}|^{2-s} \big)^{\frac{(N-2)p}{2-s}}|\frac{x}{\varepsilon}-\frac{y}{\varepsilon}|^\alpha \big(k +|\frac{y}{\varepsilon}|^{2-s} \big)^{\frac{(N-2)p}{2-s}}}d\frac{x}{\varepsilon}d\frac{y}{\varepsilon} \\
&=C\varepsilon^{2N-\alpha-(N-2)p}\int_{B_{\delta/\varepsilon}}\int_{B_{\delta/\varepsilon}}\frac{1}{ \big(k +|x|^{2-s} \big)^{\frac{(N-2)p}{2-s}}|x-y|^\alpha \big(k +|y|^{2-s} \big)^{\frac{(N-2)p}{2-s}}}dxdy\\
&=O(\varepsilon^{2N-\alpha-(N-2)p}),
\end{aligned}$$
where
$$\begin{aligned}
&\int_{B_{\delta/\varepsilon}}\int_{B_{\delta/\varepsilon}}\frac{1}{ \big(k +|x|^{2-s} \big)^{\frac{(N-2)p}{2-s}}|x-y|^\alpha \big(k +|y|^{2-s} \big)^{\frac{(N-2)p}{2-s}}}dxdy\\
&\leq C(N,\alpha) \bigg(\int_{|x|<\delta/\varepsilon}\frac{1}{ \big(k +|x|^{2-s} \big)^{\frac{(N-2)p}{2-s} \cdot \frac{2N}{2N-\alpha} }  }dx\bigg)^{\frac{2N-\alpha}{2N}}\\
&\leq \bigg(\int_{|x|<\delta/\varepsilon}\frac{r^{N-1}}{ \big(k +|r|^{2-s} \big)^{\frac{(N-2)p}{2-s} \cdot \frac{2N}{2N-\alpha} }  }dr\bigg)^{\frac{2N-\alpha}{2N}}
\leq C.
\end{aligned}$$
For any $t>0$,  we have
\begin{equation*}
\begin{aligned}
\mathcal{I}(tu_{s,\varepsilon})&= \frac{t^2}{2} \int_{\Omega} |\nabla u_{s,\varepsilon}|^2dx -\frac{t^{2p}\lambda}{2p} \int_{\Omega}\int_{\Omega} \frac{|u_{s,\varepsilon}(x)|^{p}|u_{s,\varepsilon}(y)|^{p} }{|x-y|^{\alpha}}dxdy-  \frac{t^{2^*(s)}\mu}{2^*(s)} \int_{\Omega} \frac{|u_{s,\varepsilon}|^{2^*(s)}}{|x|^s}dx\\
&\leq \frac{t^2}{2} \Big( \mu_{s} (\mathbb{R}^{N})^{\frac{N-s}{2-s}} + O(\varepsilon^{N-2})\Big)- \frac{t^{2p}\lambda}{2p} O(\varepsilon^{2N-\alpha-(N-2)p}) - \frac{t^{2^*(s)}\mu}{2^*(s)} \Big( \mu_{s} (\mathbb{R}^{N})^{\frac{N-s}{2-s}} + O(\varepsilon^{N-s}) \Big)\\
&=:h(t).
\end{aligned}
\end{equation*}
Since $1<p<2_\alpha^*$ and $2< q =2^*(s)$, then it is easy to see that $h(0)=0$ and $\lim\limits_{t\to+\infty }h(t)=-\infty$. Thus there exists $t_\varepsilon>0$ such that $h(t_\varepsilon)=\max\limits_{t>0}h(t)$ and
\begin{equation}\label{eq3.13}
\begin{aligned}
t_\varepsilon\Big(\mu_{s} (\mathbb{R}^{N})^{\frac{N-s}{2-s}} + O(\varepsilon^{N-2})\Big) -\lambda t_\varepsilon^{2p-1} O(\varepsilon^{2N-\alpha-(N-2)p})  -\mu t_\varepsilon^{2^*(s)-1}\Big( \mu_{s} (\mathbb{R}^{N})^{\frac{N-s}{2-s}} + O(\varepsilon^{N-s}) \Big)=0.
\end{aligned}
\end{equation}
For $\varepsilon>$ sufficiently small, we deduce from \eqref{eq3.13} that there exists a $t_0>0$ independent of $\varepsilon$ such that
\begin{equation*}
\begin{aligned}
t_\varepsilon&=\Bigg(\frac{\mu_{s} (\mathbb{R}^{N})^{\frac{N-s}{2-s}} + O(\varepsilon^{N-2}) -\lambda t^{2p-2} O(\varepsilon^{2N-\alpha-(N-2)p})}{\mu \mu_{s} (\mathbb{R}^{N})^{\frac{N-s}{2-s}} + O(\varepsilon^{N-s}) } \Bigg)^{\frac{1}{2^*(s)-2}}\\
&\leq\Bigg(\frac{\mu_{s} (\mathbb{R}^{N})^{\frac{N-s}{2-s}} + O(\varepsilon^{N-2}) }{\mu\mu_{s} (\mathbb{R}^{N})^{\frac{N-s}{2-s}} + O(\varepsilon^{N-s}) } \Bigg)^{\frac{1}{2^*(s)-2}}\\
&\leq \Big(\frac{1}{\mu}+ O(\varepsilon^{N-2})   \Big)^{\frac{1}{2^*(s)-2}}<t_0.
\end{aligned}
\end{equation*}
On the other hand,  by \eqref{eq3.13} again, there exists $t_1 > 0$ independent of $\varepsilon$ such that for $\varepsilon>0$,
\begin{equation*}
\begin{aligned}
t_\varepsilon&=\Bigg(\frac{\mu_{s} (\mathbb{R}^{N})^{\frac{N-s}{2-s}} + O(\varepsilon^{N-2}) -\lambda t^{2p-2} O(\varepsilon^{2N-\alpha-(N-2)p})}{\mu \mu_{s} (\mathbb{R}^{N})^{\frac{N-s}{2-s}} + O(\varepsilon^{N-s}) } \Bigg)^{\frac{1}{2^*(s)-2}}\\
&\geq \Bigg(\frac{\mu_{s} (\mathbb{R}^{N})^{\frac{N-s}{2-s}} + O(\varepsilon^{N-2}) - O(\varepsilon^{2N-\alpha-(N-2)p})}{\mu \mu_{s} (\mathbb{R}^{N})^{\frac{N-s}{2-s}} + O(\varepsilon^{N-s}) } \Bigg)^{\frac{1}{2^*(s)-2}}\\
&>\Big(\frac{1}{\mu}  -O(\varepsilon^{2N-\alpha-(N-2)p}) \Big)^{\frac{1}{2^*(s)-2}}>t_1.
\end{aligned}
\end{equation*}
Thus we have a lower and a upper bound for $t_\varepsilon$, independent of $\varepsilon$. Therefore, we get that
\begin{equation}\label{eq05.2}
\begin{aligned}
\max_{t\geq0}\mathcal{I}(tu_{s,\varepsilon})
&\leq \Big(\frac{1}{2}-\frac{1}{2^*(s)} \Big) \Bigg( \frac{ \int_{\Omega} |\nabla u_{\varepsilon}|^2dx}
{\Big(\mu \int_{\Omega} \frac{|u_{\varepsilon}|^{2^*(s)}}{|x|^s}dx \Big)^{\frac{2}{2^*(s)}}}\Bigg)^{\frac{2^*(s)}{2^*(s)-2}}-\lambda  O(\varepsilon^{2N-\alpha-(N-2)p})\\
& =\Big(\frac{1}{2}-\frac{1}{2^*(s)} \Big) \Bigg( \frac{\mu_{s} (\mathbb{R}^{N})^{\frac{N-s}{2-s}} + O(\varepsilon^{N-2})}
{\Big(\mu  \mu_{s} (\mathbb{R}^{N})^{\frac{N-s}{2-s}} + O(\varepsilon^{N-s}) \Big)^{\frac{2}{2^*(s)}}}\Bigg)^{\frac{2^*(s)}{2^*(s)-2}}-\lambda  O(\varepsilon^{2N-\alpha-(N-2)p})\\
&=\Big(\frac{1}{2}-\frac{1}{2^*(s)} \Big) \bigg(\frac{1}{\mu^{\frac{N-2}{N-s}}} \mu_{s} (\mathbb{R}^{N})  + O(\varepsilon^{N-2})\bigg)^{\frac{N-s}{2-s}}  -\lambda  O(\varepsilon^{2N-\alpha-(N-2)p})\\
& \leq \frac{2-s}{2(N-s)}\frac{1}{\mu^{\frac{N-2}{2-s}}} \mu_{s} (\mathbb{R}^{N})^{\frac{N-s}{2-s}}  + O(\varepsilon^{N-2})  -\lambda O(\varepsilon^{2N-\alpha-(N-2)p}) \\
&<\frac{2-s}{2(N-s)} \frac{1}{\mu^{\frac{2}{2^*(s)-2}}} \mu_{s}(\mathbb{R}^N)^{\frac{N-s}{2-s}}.
\end{aligned}
\end{equation}
Now we distinguish the following two cases:
\begin{itemize}
\item[(\romannumeral1)] In the case $2_\alpha^*>p>2_\alpha^*=\frac{2N-\alpha}{N-2}-1$, we know that \(2N-\alpha-(N-2)p<N-2 \). Combing this with \eqref{eq05.2},  we get
    $$\max_{t\geq0}\mathcal{I}(tu_{s,\varepsilon})\leq\frac{2-s}{2(N-s)}\frac{1}{\mu^{\frac{N-2}{2-s}}} \mu_{s} (\mathbb{R}^{N})^{\frac{N-s}{2-s}}  - O(\varepsilon^{2N-\alpha-(N-2)p}). $$ Thus, we get the conclusion for $\varepsilon$ small enough.
\item[(\romannumeral2)] In the case $1<p\leq2_\alpha^*=\frac{2N-\alpha}{N-2}-1$ with $\alpha\leq4$, we know that \(2N-\alpha-(N-2)p\geq N-2 \).  By \eqref{eq05.2}, we have $$\max_{t\geq0}\mathcal{I}(tu_{s,\varepsilon}) \leq\frac{2-s}{2(N-s)}\frac{1}{\mu^{\frac{N-2}{2-s}}} \mu_{s} (\mathbb{R}^{N})^{\frac{N-s}{2-s}} + O(\varepsilon^{N-2}) - \lambda O(\varepsilon^{2N-\alpha-(N-2)p}), $$
for $\lambda=\varepsilon^{-\theta}$ with $\theta>2N-\alpha-(N-2)(p+1)$, we also get the conclusion.
\end{itemize}

\medskip

$(2)$ Let $p=2_\alpha^*$ and  $2< q <2^*(s)$. It follows from  \cite[Theorem~ 1.42]{Willem1996} that  $U(x)=\frac{[N(N-2)]^{\frac{N-2}{4}}}{(1+|x|^2)^{\frac{N-2}{2}}}$ is a minimizer for $S$.
Let $\phi\in C_0^\infty(\Omega)$ be a nonnegative function such that
$\phi(x)=1,~~x\in B(0,\rho)$, $\phi(x)=0$, $x\in \Omega\setminus B(0,2\rho)$ and $0\leq \phi(x)\leq 1$, $x\in \Omega$.
For $\varepsilon>0$, we define,
$$\begin{aligned}
u_\varepsilon(x):=\phi(x)U_\varepsilon(x) ~~\text{where}~~U_\varepsilon(x) :=\varepsilon^{\frac{2-N}{2}}U\Big(\frac{x}{\varepsilon}\Big).\end{aligned}$$
It follows from   \cite{Gao-Yang2018SCM} and  \cite{Willem1996} that
$$|\nabla U_\varepsilon|_2^2=|U_\varepsilon|_{2^*}^{2^*}=S^{\frac N2},$$
$$\int_\Omega|\nabla u_\varepsilon|^2dx=S^{\frac N2}+O(\varepsilon^{N-2})=C(N,\alpha)^{\frac{N(N-2)}{2(2N-\alpha)}}S_{H,L}^{\frac N2}+O(\varepsilon^{N-2})$$
and
$$\displaystyle{\int_{\Omega}}\displaystyle{\int_{\Omega}} \frac{|u_\varepsilon(x)|^{2_{\alpha}^*}|u_\varepsilon(y)|^{2_{\alpha}^*} }{|x-y|^{\alpha}}dxdy \geq C(N,\alpha)^{\frac{N}{2}} S_{H,L}^{\frac{2N-\alpha}{2}} -O(\varepsilon^{\frac{2N-\alpha}{2}}) = C(N,\alpha) S^{\frac{2N-\alpha}{2}} -O(\varepsilon^{\frac{2N-\alpha}{2}}).$$
Direct computations show that
\begin{equation}\label{eq2.3002}
\begin{aligned}
\int_{\Omega} \frac{|u_\varepsilon|^{q}}{|x|^s}dx=
\begin{cases}
&C\varepsilon^{N-\frac{(N-2)q}{2}-s},~~~~~~~~\text{if }~q> \frac{N-s}{N-2},\\
&C\varepsilon^{N-\frac{(N-2)q}{2}-s}|ln\varepsilon |,~~\text{if }~q= \frac{N-s}{N-2}, \\&C\varepsilon^{\frac{N-2}{2}q},~~~~~~~~~~~~~~~~\text{if }~ q= \frac{N-s}{N-2}.
\end{cases}
\end{aligned}
\end{equation}
Indeed, on one hand,
\begin{equation*}
\begin{aligned}
\int_{\Omega} \frac{|u_\varepsilon|^{q}}{|x|^s}dx &= \int_{B_{2\rho}}\frac{|\phi(x)U_\varepsilon(x)|^{q}}{|x|^s}dx  \geq \int_{B_{\rho}}\frac{|U_\varepsilon(x)|^{q}}{|x|^s}dx\\
&= C\int_{B_\rho} \frac{\varepsilon^{\frac{(2-N)q}{2}}}{|x|^s\big(1+|\frac{x}{\varepsilon}|^2 \big)^{\frac{(N-2)q}{2}}} dx\\
&=C\int_{B_\rho} \frac{\varepsilon^{\frac{(2-N)q}{2}-s}}{|\frac{x}{\varepsilon}|^s \big(1+|\frac{x}{\varepsilon}|^2 \big)^{\frac{(N-2)q}{2}}} dx\\
&=C\int_{|x|<\rho/\varepsilon} \frac{\varepsilon^{N-\frac{(N-2)q}{2}-s}}{|x|^s \big(1+|x|^2 \big)^{\frac{(N-2)q}{2}}} dx\\
&=C \varepsilon^{N-\frac{(N-2)q}{2}-s} \int_0^{\rho/\varepsilon}\frac{r^{N-1}}{r^s(1+r^2)^{\frac{(N-2)q}{2}} }dr\\
&\geq C \varepsilon^{N-\frac{(N-2)q}{2}-s} \int_1^{\rho/\varepsilon}\frac{r^{N-1}}{r^s(1+r^2)^{\frac{(N-2)q}{2}} }dr\\
& \geq C \varepsilon^{N-\frac{(N-2)q}{2}-s}  \int_1^{\rho/\varepsilon}  r^{N-s-(N-2)q-1}dr \\
&=\begin{cases}
&C\varepsilon^{N-\frac{(N-2)q}{2}-s},~~~~~~~~\text{if }~q> \frac{N-s}{N-2},\\
&C\varepsilon^{N-\frac{(N-2)q}{2}-s}|ln\varepsilon |,~~\text{if }~q= \frac{N-s}{N-2}, \\&C\varepsilon^{\frac{N-2}{2}q},~~~~~~~~~~~~~~~~\text{if }~q< \frac{N-s}{N-2}.
\end{cases}\\
\end{aligned}
\end{equation*}

On the orther hand,

\begin{equation*}
\begin{aligned}
\int_{\Omega} \frac{|u_\varepsilon|^{q}}{|x|^s}dx &\leq \int_{B_{2\rho}}\frac{|U_\varepsilon(x)|^{q}}{|x|^s}dx\\
&\leq C \varepsilon^{N-\frac{(N-2)q}{2}-s} \int_0^{\rho/\varepsilon}\frac{r^{N-s-1}}{(1+r^2)^{\frac{(N-2)q}{2}} }dr\\
&= C \varepsilon^{N-\frac{(N-2)q}{2}-s} \int_0^{1}\frac{r^{N-s-1}}{(1+r^2)^{\frac{(N-2)q}{2}} }dr +C \varepsilon^{N-\frac{(N-2)q}{2}-s} \int_1^{\rho/\varepsilon}\frac{r^{N-s-1}}{(1+r^2)^{\frac{(N-2)q}{2}} }dr \\
&\leq C \varepsilon^{N-\frac{(N-2)q}{2}-s}  +C \varepsilon^{N-\frac{(N-2)q}{2}-s} \int_1^{\rho/\varepsilon}  r^{N-s-(N-2)q-1}dr  \\
&\leq\begin{cases}
&C\varepsilon^{N-\frac{(N-2)q}{2}-s},~~~~~~~~\text{if }~q> \frac{N-s}{N-2},\\
&C\varepsilon^{N-\frac{(N-2)q}{2}-s}|ln\varepsilon |,~~\text{if }~q= \frac{N-s}{N-2}, \\&C\varepsilon^{\frac{N-2}{2}q},~~~~~~~~~~~~~~~~\text{if }~ q= \frac{N-s}{N-2}.
\end{cases}
\end{aligned}
\end{equation*}
For any $t>0$,  we have
\begin{equation*}
\begin{aligned}
\mathcal{I}(tu_\varepsilon)
&\leq\frac{t^2}{2}\Big(S^{\frac N2}+O(\varepsilon^{N-2})\Big) -\frac{\lambda t^{2\cdot{2_{\alpha}^*}}}{2\cdot{2_{\alpha}^*}}\Big( C(N,\alpha) S^{\frac{2N-\alpha}{2}} -O(\varepsilon^{\frac{2N-\alpha}{2}}) \Big)-  \frac{\mu t^q}{q} \int_{\Omega} \frac{|u_\varepsilon|^{q}}{|x|^s}dx  \\
&=:\bar{h}(t).
\end{aligned}
\end{equation*}
Since $2< q <2^*(s)$, then it is easy to see that $\bar{h}(0)=0$ and $\lim\limits_{t\to+\infty }\bar{h}(t)=-\infty$. Thus there exists $t_\varepsilon>0$ such that $\bar{h}(t_\varepsilon)=\max\limits_{t>0}\bar{h}(t)$ and
\begin{equation}\label{eq3.16}
\begin{aligned}
t_\varepsilon\Big(S^{\frac N2}+O(\varepsilon^{N-2})\Big)- \lambda t_\varepsilon^{2\cdot{2_{\alpha}^*}-1}\Big(  C(N,\alpha) S^{\frac{2N-\alpha}{2}} -O(\varepsilon^{\frac{2N-\alpha}{2}})\Big)  -\mu t_\varepsilon^{q-1} \int_{\Omega} \frac{|u_\varepsilon|^{q}}{|x|^s}dx  =0.
\end{aligned}
\end{equation}
For $\varepsilon>$ small enough, we also deduce from \eqref{eq3.16} that there exists a $t_2>0$ independent of $\varepsilon$ such that
\begin{equation*}
\begin{aligned}
t_\varepsilon&=\Bigg(\frac{S^{\frac N2}+O(\varepsilon^{N-2}) - \mu t_\varepsilon^{q-2} \int_{\Omega} \frac{|u_\varepsilon|^{q}}{|x|^s}dx}{ C(N,\alpha) S^{\frac{2N-\alpha}{2}} -O(\varepsilon^{\frac{2N-\alpha}{2}})} \Bigg)^{\frac{1}{2\cdot{2_{\alpha}^*}-2}}\\
&\leq\Bigg(\frac{S^{\frac N2}+O(\varepsilon^{N-2}) }{ C(N,\alpha) S^{\frac{2N-\alpha}{2}} -O(\varepsilon^{\frac{2N-\alpha}{2}})} \Bigg)^{\frac{1}{2\cdot{2_{\alpha}^*}-2}}\\
&\leq \Big(\frac{1}{C(N,\alpha) S^{\frac{N-\alpha}{2}} }+O(\varepsilon^{(N-2)})   \Big)^{\frac{1}{2\cdot{2_{\alpha}^*}-2}}<t_2.
\end{aligned}
\end{equation*}
On the other hand, it follows from \eqref{eq3.16} that there exists $t_3 > 0$ independent of $\varepsilon$ such that for $\varepsilon>0$,
\begin{equation*}
\begin{aligned}
t_\varepsilon&=\Bigg(\frac{S^{\frac N2}+O(\varepsilon^{N-2}) - \mu t_\varepsilon^{q-2} \int_{\Omega} \frac{|u_\varepsilon|^{q}}{|x|^s}dx}{ C(N,\alpha) S^{\frac{2N-\alpha}{2}} -O(\varepsilon^{\frac{2N-\alpha}{2}})} \Bigg)^{\frac{1}{2\cdot{2_{\alpha}^*}-2}}\\
&\geq\Bigg(\frac{S^{\frac N2} - \mu t_\varepsilon^{q-2} \int_{\Omega} \frac{|u_\varepsilon|^{q}}{|x|^s}dx }{ C(N,\alpha) S^{\frac{2N-\alpha}{2}} -O(\varepsilon^{\frac{2N-\alpha}{2}})} \Bigg)^{\frac{1}{2\cdot{2_{\alpha}^*}-2}}\\
&\geq \Big(\frac{1}{2C(N,\alpha) S^{\frac{N-\alpha}{2}} }   \Big)^{\frac{1}{2\cdot{2_{\alpha}^*}-2}}>t_3.
\end{aligned}
\end{equation*}
Thus we have a lower and a upper bound for $t_\varepsilon$, independent of $\varepsilon$. We can argue as in $(1)$ to conclude that
$$\begin{aligned}
0<\max_{t\geq0}\mathcal{I}(tu_\varepsilon)&\leq \Big(\frac{1}{2}-\frac{1}{2\cdot{2_{\alpha}^*}} \Big) \Bigg( \frac{\int_\Omega|\nabla u_\varepsilon|^2dx}
{\Big(\lambda\int_\Omega\int_\Omega\frac{|u_\varepsilon(x)|^{2_\alpha^*} |u_\varepsilon|^{2_\alpha^*}}{|x-y|^\alpha}dxdy \Big)^{\frac{1}{2_\alpha^*}}}\Bigg)^{\frac{2_\alpha^*}{2_\alpha^*-1}}-  \frac{\mu t_0^{q}}{q} \int_{\Omega} \frac{|u|^{q}}{|x|^s}dx\\
& \leq \frac{N-\alpha+2}{2(2N-\alpha)} \Bigg(\frac{ S^{\frac N2}+O(\varepsilon^{N-2}) }
{\Big(\lambda C(N,\alpha) S^{\frac{2N-\alpha}{2}} -O(\varepsilon^{\frac{2N-\alpha}{2}}) \Big)^\frac{N-2}{2N-\alpha}}  \Bigg)^{\frac{2_\alpha^*}{2_\alpha^*-1}} -\frac{\mu t_0^{q}}{q} \int_{\Omega} \frac{|u|^{q}}{|x|^s}dx\\
& \leq \frac{N-\alpha+2}{2(2N-\alpha)} \Bigg( \frac{S^{\frac N2}}{\big(\lambda C(N,\alpha) S^{\frac{2N-\alpha}{2}} \big)^{\frac{1}{2_\alpha^*}}} +O(\varepsilon^{N-2}) \Bigg)^{\frac{2_\alpha^*}{2_\alpha^*-1}} -\frac{\mu t_0^{q}}{q} \int_{\Omega} \frac{|u|^{q}}{|x|^s}dx\\
&\leq \frac{N-\alpha+2}{2(2N-\alpha)}  \Big(\frac{1}{\lambda C(N,\alpha) }  \Big)^{\frac{1}{2_\alpha^*-1}} S^{\frac{2_\alpha^*}{2_\alpha^*-1}}
 +O(\varepsilon^{N-2}) -\frac{\mu t_0^{q}}{q} \int_{\Omega} \frac{|u|^{q}}{|x|^s}dx.
\end{aligned}$$

Now we distinguish the following cases.
\begin{itemize}
\item[(\romannumeral1)] In the case $N=3$ and $s<1$, we know that \(2<\frac{N-s}{N-2}<\frac{2(2-s)}{N-2}=\frac{2(N-s)}{N-2}-2 <2^*(s)\). If $q> \frac{2(2-s)}{N-2}$,    by \eqref{eq2.3002} we get
    $$\max_{t\geq0}\mathcal{I}(tu_{\varepsilon})\leq \frac{N-\alpha+2}{2(2N-\alpha)}  \Big(\frac{1}{\lambda C(N,\alpha) }  \Big)^{\frac{1}{2_\alpha^*-1}} S^{\frac{2_\alpha^*}{2_\alpha^*-1}}  +O(\varepsilon^{N-2}) -O(\varepsilon^{N-\frac{(N-2)q}{2}-s}). $$
    Since $q> \frac{2(2-s)}{N-2}$, we have that $N-\frac{(N-2)q}{2}-s<N-2$, and we get the conclusion for $\varepsilon$ small enough.

    If $\frac{N-s}{N-2}< q\leq \frac{2(2-s)}{N-2}$,  by \eqref{eq2.3002}  we get
    $$ \max_{t\geq0}\mathcal{I}(tu_{\varepsilon})\leq \frac{N-\alpha+2}{2(2N-\alpha)}  \Big(\frac{1}{\lambda C(N,\alpha) }  \Big)^{\frac{1}{2_\alpha^*-1}} S^{\frac{2_\alpha^*}{2_\alpha^*-1}} +O(\varepsilon^{N-2}) -\mu O(\varepsilon^{N-\frac{(N-2)q}{2}-s}). $$
     Since $\frac{N-s}{N-2}< q\leq \frac{2(2-s)}{N-2}$ means that $N-\frac{(N-2)q}{2}-s\geq N-2$, we also get the conclusion with $\mu=\varepsilon^{-\theta}$ with $\theta>N-\frac{(N-2)q}{2}-s-N-2 $ for $\varepsilon$ small enough.

     If $q=\frac{N-s}{N-2}$,   by \eqref{eq2.3002} we get
    $$ \max_{t\geq0}\mathcal{I}(tu_{\varepsilon})\leq \frac{N-\alpha+2}{2(2N-\alpha)}  \Big(\frac{1}{\lambda C(N,\alpha) }  \Big)^{\frac{1}{2_\alpha^*-1}} S^{\frac{2_\alpha^*}{2_\alpha^*-1}}  +O(\varepsilon^{N-2}) - \mu O(\varepsilon^{\frac{N-s}{2}})|ln \varepsilon|. $$
     Since
     $$\begin{aligned}
     \lim_{\varepsilon\to 0}\frac{\varepsilon^{\frac{N-s}{2}} |ln \varepsilon|}{\varepsilon^{N-2}}=\lim_{\varepsilon\to 0}\frac{-1}{ \big(N-2 - \frac{N-s}{2}\big)\varepsilon^{N-2- \frac{N-s}{2}} }=0,
     \end{aligned}$$
     we again get the conclusion  with $\mu=\varepsilon^{-\theta}$ and $N-2- \frac{N-s}{2}+\theta\in (0,1)$ for $\varepsilon$ small enough.

    If $2<q<\frac{N-s}{N-2}$,   by \eqref{eq2.3002} we see that
    $$ \max_{t\geq0}\mathcal{I}(tu_{\varepsilon})\leq \frac{N-\alpha+2}{2(2N-\alpha)}  \Big(\frac{1}{\lambda C(N,\alpha) }  \Big)^{\frac{1}{2_\alpha^*-1}} S^{\frac{2_\alpha^*}{2_\alpha^*-1}} +O(\varepsilon^{N-2}) -\mu  O(\varepsilon^{\frac{N-2}{2}q}), $$
    and for $\mu=\varepsilon^{-\theta}$ with $\theta> \frac{(N-2)(q-2)}{2}$, we also get the conclusion.
\end{itemize}

\smallskip

\begin{itemize}
\item[(\romannumeral2)]  In the case $N=3$ and $s=1$, we know that \(\frac{N-s}{N-2}=2 \).     By \eqref{eq2.3002} and  $2<q< \frac{2(N-s)}{N-2}$, we get
    $$ \max_{t\geq0}\mathcal{I}(tu_{\varepsilon})\leq \frac{N-\alpha+2}{2(2N-\alpha)}  \Big(\frac{1}{\lambda C(N,\alpha) }  \Big)^{\frac{1}{2_\alpha^*-1}} S^{\frac{2_\alpha^*}{2_\alpha^*-1}} +O(\varepsilon^{N-2}) -O(\varepsilon^{N-\frac{(N-2)q}{2}-s}). $$
    In view of $N-\frac{(N-2)q}{2}-s<N-2$, we get the conclusion for $\varepsilon$ small enough.
\end{itemize}

\smallskip

\begin{itemize}
\item[(\romannumeral3)] In the case $N=3$ and $1<s<2$ or $N\geq4$ and $0<s<2$, we know that \(\frac{N-s}{N-2}<2 \).     By \eqref{eq2.3002} and  $2<q< \frac{2(2-s)}{N-2}$, we get
    $$ \max_{t\geq0}\mathcal{I}(tu_{\varepsilon})\leq \frac{N-\alpha+2}{2(2N-\alpha)}  \Big(\frac{1}{\lambda C(N,\alpha) }  \Big)^{\frac{1}{2_\alpha^*-1}} S^{\frac{2_\alpha^*}{2_\alpha^*-1}} +O(\varepsilon^{N-2}) -O(\varepsilon^{N-\frac{(N-2)q}{2}-s}). $$
    In view of $N-\frac{(N-2)q}{2}-s<N-2$, we get again the conclusion for $\varepsilon$ small enough.
\end{itemize}

\end{proof}

\medskip

We can now finally derive the main existence result.

\noindent\textbf{Proof of Theorem \ref{Thm1.1}.}
From Lemma \ref{Lem2.1} and the mountain pass theorem \cite{Willem1996},
there exists a $(PS)_c$ sequence $\{u_n\}$ such that
\begin{equation}
\mathcal{I}(u_n)\to c~~\text{ and}~~~\mathcal{I}'(u_n)\to 0~\text{in }~ (H_0^1(\Omega))^{-1}
\end{equation}
where
$$c=\inf_{\gamma\in \Gamma}\sup_{t\in[0,1]}\mathcal{I}(\gamma(t))$$
and $\Gamma$ is defined by
$$\Gamma:=\Big\{\gamma\in C([0,1], H_0^1(\Omega)): \gamma(0)=0~~\text{and}~~\mathcal{I}(\gamma(1))<0  \Big\}.$$
By Lemma \ref{Lem2.2}, if one of the following conditions is satisfied:
\begin{itemize}
\item[$(1)$] $\lambda, \mu>0$, $1<p<2_\alpha^*$ and  $2< q < 2^*(s)$,
\item[$(2)$] $\lambda>0$, $0<\mu<\bar{\mu}$,  $1<p<2_\alpha^*$ and $2= q =2^*(s)$( that is, $s=2$),
\end{itemize}
then the sequence $\{u_n\}$ contains a strongly convergent subsequence and $\mathcal{I}(u_0)=c>0$. This yields a nontrivial solution for \eqref{eq1.1}.

\medskip
On the other hand, by Lemma \ref{Lem2.3} we obtain that
\begin{itemize}
\item[$(3)$] if $\lambda, \mu>0$,  $2< q =2^*(s)$ and one of the following conditions is satisfied:
    \begin{itemize}
    \item[(\romannumeral1)] $2_\alpha^*-1<p<2_\alpha^*$ and $\lambda>0$,
    \item[(\romannumeral2)] $1<p\leq2_\alpha^*-1$ with $\alpha\leq4$ and $\lambda>0$ is sufficiently large,
    \end{itemize}
    then   \[c<\frac{2-s}{2(N-s)} \frac{1}{\mu^{\frac{2}{2^*(s)-2}}} \mu_{s}(\mathbb{R}^N)^{\frac{N-s}{2-s}}.\]
\item[$(4)$] if  $\lambda, \mu>0$, $p=2_\alpha^*$, $2< q <2^*(s)$  and one of the following conditions id satisfied:
    \begin{itemize}
     \item[(\romannumeral1)] $N=3$, $s<1$, $2^*(s)-2<q< 2^*(s)$ and $\mu>0$,
     \item[(\romannumeral2)] $N=3$, $s<1$, $2< q\leq 2^*(s)-2$ and $\mu>0$ is sufficiently large,
     \item[(\romannumeral3)] $N=3$, $1\leq s<2$, $2<q< 2^*(s)$ and $\mu>0$,
     \item[(\romannumeral4)] $N\geq4$, $0<s<2$, $2<q< 2^*(s)$ and $\mu>0$,
    \end{itemize}
  then
    \[c< \frac{N-\alpha+2}{2(2N-\alpha)}  \Big(\frac{1}{\lambda C(N,\alpha) }  \Big)^{\frac{1}{2_\alpha^*-1}} S^{\frac{2_\alpha^*}{2_\alpha^*-1}}.\]
\end{itemize}
Together with Lemma \ref{Lem2.2}, we again deduce that the sequence $\{u_n\}$ has a strongly convergent subsequence and equation \eqref{eq1.1} has a nontrivial solution.
\raisebox{-0.5mm}{\rule{2.5mm}{3mm}}\vspace{6pt}

\section*{Acknowledgments}
This work was supported by Xingdian talent support program of Yunnan Province, National Natural Science Foundation of China (12261107, 12101546), Yunnan Fundamental Research Projects (202301AU070144, 202301AU070159), Scientific Research Fund of Yunnan Educational Commission (2023J0199, 2023Y0515). A.J. is partially supported by PRIN 2022 "\emph{Pattern formation in nonlinear phenomena}" and is a member of the INDAM Research Group GNAMPA.

\bibliographystyle{plain}
\bibliography{gu}

\begin{thebibliography}{10}

\bibitem{Bahrami2014}
M.~Bahrami, A.~Gro\ss~ardt, S.~Donadi, and A.~Bassi.
\newblock The {S}chr\"{o}dinger-{N}ewton equation and its foundations.
\newblock {\em New J. Phys.}, 16(November):115007, 17, 2014.

\bibitem{Brezis-Nirenberg1983CPAM}
H.~Br\'{e}zis and L.~Nirenberg.
\newblock Positive solutions of nonlinear elliptic equations involving critical
  {S}obolev exponents.
\newblock {\em Comm. Pure Appl. Math.}, 36(4):437--477, 1983.

\bibitem{Caffarelli-Kohn-Nirenberg1984CM}
L.~Caffarelli, R.~Kohn, and L.~Nirenberg.
\newblock First order interpolation inequalities with weights.
\newblock {\em Compositio Math.}, 53(3):259--275, 1984.

\bibitem{Cerami-Zhong-Zou2015CVPDE}
G.~Cerami, X.~Zhong, and W.~Zou.
\newblock On some nonlinear elliptic {PDE}s with {S}obolev-{H}ardy critical
  exponents and a {L}i-{L}in open problem.
\newblock {\em Calc. Var. Partial Differential Equations}, 54(2):1793--1829,
  2015.

\bibitem{Chen-Radulescu-Shu-Wei2025MA}
S.~Chen, V.~R\u{a}dulescu, M.~Shu, and J.~Wei.
\newblock Static solutions for {C}hoquard equations with {C}oulomb potential
  and upper critical growth.
\newblock {\em Math. Ann.}, 392(2):2081--2130, 2025.

\bibitem{Chen-Wang2024CVPDE}
We. Chen and Z.~Wang.
\newblock Blowing-up solutions for a slightly subcritical {C}hoquard equation.
\newblock {\em Calc. Var. Partial Differential Equations}, 63(9):Paper No. 235,
  36, 2024.

\bibitem{Choquard-Stubbe-Vuffray2008DIE}
P.~Choquard, J.~Stubbe, and M.~Vuffray.
\newblock Stationary solutions of the {S}chr\"{o}dinger-{N}ewton model---an
  {ODE} approach.
\newblock {\em Differential Integral Equations}, 21(7-8):665--679, 2008.

\bibitem{Cingolani-Clapp-Secchi2012ZAMP}
S.~Cingolani, M.~Clapp, and S.~Secchi.
\newblock Multiple solutions to a magnetic nonlinear {C}hoquard equation.
\newblock {\em Z. Angew. Math. Phys.}, 63(2):233--248, 2012.

\bibitem{Gao-Yang2017JMAA}
F.~Gao and M.~Yang.
\newblock On nonlocal {C}hoquard equations with {H}ardy-{L}ittlewood-{S}obolev
  critical exponents.
\newblock {\em J. Math. Anal. Appl.}, 448(2):1006--1041, 2017.

\bibitem{Gao-Yang2018SCM}
F.~Gao and M.~Yang.
\newblock The {B}rezis-{N}irenberg type critical problem for the nonlinear
  {C}hoquard equation.
\newblock {\em Sci. China Math.}, 61(7):1219--1242, 2018.

\bibitem{Ghoussoub-Kang2004AIHP}
N.~Ghoussoub and X.~Kang.
\newblock Hardy-{S}obolev critical elliptic equations with boundary
  singularities.
\newblock {\em Ann. Inst. H. Poincar\'{e} C Anal. Non Lin\'{e}aire},
  21(6):767--793, 2004.

\bibitem{Ghoussoub-Robert2006GFA}
N.~Ghoussoub and F.~Robert.
\newblock The effect of curvature on the best constant in the {H}ardy-{S}obolev
  inequalities.
\newblock {\em Geom. Funct. Anal.}, 16(6):1201--1245, 2006.

\bibitem{Ghoussoub-Yuan2000TAMS}
N.~Ghoussoub and C.~Yuan.
\newblock Multiple solutions for quasi-linear {PDE}s involving the critical
  {S}obolev and {H}ardy exponents.
\newblock {\em Trans. Amer. Math. Soc.}, 352(12):5703--5743, 2000.

\bibitem{Guo-Hu-Peng-Shuai2019CVPDE}
L.~Guo, T.~Hu, S.~Peng, and W.~Shuai.
\newblock Existence and uniqueness of solutions for {C}hoquard equation
  involving {H}ardy-{L}ittlewood-{S}obolev critical exponent.
\newblock {\em Calc. Var. Partial Differential Equations}, 58(4):Paper No. 128,
  34, 2019.

\bibitem{Guo-Tang2025arXiv}
T.~Guo and Tang X.
\newblock Existence and qualitative properties of solutions for a choquard-type
  equation with hardy potential.
\newblock {\em arXiv: 2312.11855}.

\bibitem{He2022JMAA}
R.~He.
\newblock Infinitely many solutions for the {B}r\'{e}zis-{N}irenberg problem
  with nonlinear {C}hoquard equations.
\newblock {\em J. Math. Anal. Appl.}, 515(2):Paper No. 126426, 24, 2022.

\bibitem{Hsia-Lin-Wadade2010JFA}
C.~Hsia, C.~Lin, and H.~Wadade.
\newblock Revisiting an idea of {B}r\'{e}zis and {N}irenberg.
\newblock {\em J. Funct. Anal.}, 259(7):1816--1849, 2010.

\bibitem{Li-Lin2012ARMA}
Y.~Li and C.~Lin.
\newblock A nonlinear elliptic {PDE} and two {S}obolev-{H}ardy critical
  exponents.
\newblock {\em Arch. Ration. Mech. Anal.}, 203(3):943--968, 2012.

\bibitem{Lieb1967SAM}
E.~Lieb.
\newblock Existence and uniqueness of the minimizing solution of {C}hoquard's
  nonlinear equation.
\newblock {\em Studies in Appl. Math.}, 57(2):93--105, 1976/77.

\bibitem{Lieb1983AM}
E.~Lieb.
\newblock Sharp constants in the {H}ardy-{L}ittlewood-{S}obolev and related
  inequalities.
\newblock {\em Ann. of Math. (2)}, 118(2):349--374, 1983.

\bibitem{Lieb-Loss2001book}
E.~Lieb and M.~Loss.
\newblock {\em Analysis}, volume~14 of {\em Graduate Studies in Mathematics}.
\newblock American Mathematical Society, Providence, RI, second edition, 2001.

\bibitem{Lions1980NA}
P.~Lions.
\newblock The {C}hoquard equation and related questions.
\newblock {\em Nonlinear Anal.}, 4(6):1063--1072, 1980.

\bibitem{Liu-Yang2025JMAA}
C.~Liu and X.~Yang.
\newblock Sign-changing solutions for critical {C}hoquard equation on bounded
  domain.
\newblock {\em J. Math. Anal. Appl.}, 541(2):Paper No. 128726, 22, 2025.

\bibitem{Ma-Zhao2010ARMA}
L.~Ma and L.~Zhao.
\newblock Classification of positive solitary solutions of the nonlinear
  {C}hoquard equation.
\newblock {\em Arch. Ration. Mech. Anal.}, 195(2):455--467, 2010.

\bibitem{Moroz-Van-Schaftingen2013JFA}
V.~Moroz and J.~Van~Schaftingen.
\newblock Groundstates of nonlinear {C}hoquard equations: existence,
  qualitative properties and decay asymptotics.
\newblock {\em J. Funct. Anal.}, 265(2):153--184, 2013.

\bibitem{Moroz-VanSchaftingen2013JFA}
V.~Moroz and J.~Van~Schaftingen.
\newblock Groundstates of nonlinear {C}hoquard equations: existence,
  qualitative properties and decay asymptotics.
\newblock {\em J. Funct. Anal.}, 265(2):153--184, 2013.

\bibitem{Pan-Wen-Yang2025JDE}
K.~Pan, S.~Wen, and J.~Yang.
\newblock Qualitative analysis to an eigenvalue problem of the {H}artree type
  {B}r\'{e}zis-{N}irenberg problem.
\newblock {\em J. Differential Equations}, 440(part 1):113417, 2025.

\bibitem{Pekar1954}
S.~Pekar.
\newblock Untersuchungen \"{u}ber die elektronentheorie der kristalle.
\newblock {\em Akademie Verlag}, 1954.

\bibitem{Squassina-Yang-Zhao2023}
M.~Squassina, M.~Yang, and S.~Zhao.
\newblock Local uniqueness of blow-up solutions for critical {H}artree
  equations in bounded domain.
\newblock {\em Calc. Var. Partial Differential Equations}, 62(8):Paper No. 217,
  51, 2023.

\bibitem{Talenti1976AMPA}
G.~Talenti.
\newblock Best constant in {S}obolev inequality.
\newblock {\em Ann. Mat. Pura Appl. (4)}, 110:353--372, 1976.

\bibitem{Willem1996}
M.~Willem.
\newblock {\em Minimax theorems}, volume~24 of {\em Progress in Nonlinear
  Differential Equations and their Applications}.
\newblock Birkh\"auser Boston, Inc., Boston, MA, 1996.

\bibitem{Xia-Zhang2024SIAM}
J.~Xia and X.~Zhang.
\newblock Multibump solutions for critical {C}hoquard equation.
\newblock {\em SIAM J. Math. Anal.}, 56(3):3832--3860, 2024.

\bibitem{Yang-Ye-Zhao2023JDE}
M.~Yang, W.~Ye, and S.~Zhao.
\newblock Existence of concentrating solutions of the {H}artree type
  {B}r\'{e}ezis-{N}irenberg problem.
\newblock {\em J. Differential Equations}, 344:260--324, 2023.

\bibitem{Yang-Zhao2023JGA}
M.~Yang and S.~Zhao.
\newblock Blow-up behavior of solutions to critical {H}artree equations on
  bounded domain.
\newblock {\em J. Geom. Anal.}, 33(6):Paper No. 191, 63, 2023.

\bibitem{Zhong-Zou2016CCM}
X.~Zhong and W.~Zou.
\newblock A perturbed nonlinear elliptic {PDE} with two {H}ardy-{S}obolev
  critical exponents.
\newblock {\em Commun. Contemp. Math.}, 18(4):1550061, 26, 2016.

\end{thebibliography}

\end{document}